\documentclass{gtmon_a}
\pdfoutput=1

\usepackage{amscd}

\proceedingstitle{The Zieschang Gedenkschrift}
\conferencestart{5 September 2007}
\conferenceend{8 September 2007}
\conferencename{Conference in honour of Heiner Zieschang}
\conferencelocation{Toulouse, France}

\editor{Michel Boileau}
\givenname{Michel}
\surname{Boileau}

\editor{Martin Scharlemann}
\givenname{Martin}
\surname{Scharlemann}

\editor{Richard Weidmann}
\givenname{Richard}
\surname{Weidmann}

\title{Minimizing the number of Nielsen preimage classes}

\author{Olga Frolkina}
\givenname{Olga}
\surname{Frolkina}
\address{
Chair of General Topology and Geometry\\
Faculty of Mechanics and Mathematics\\\newline
M\,V\,Lomonosov Moscow State University\\\newline
Leninskie Gori, 119991 Moscow, GSP-1\\Russia
}
\email{frolkina@mech.math.msu.su}
\email{odfrolki@mail.ru}
\urladdr{}

\dedicatory{I feel very honoured to have the possibility to contribute a paper
to this volume dedicated to the memory of
an outstanding mathematician and 
a pleasant good-humoured person: Heiner Zieschang.
In~2002--2003 in M\,V\,Lomonosov Moscow State University
Heiner gave a series of 
lectures on fixed points and coincidence theory, 
which I was lucky to attend. 
In the same period I learned the German language at his seminars. 
During a nice voyage in summer~2003 
from Moscow to Saint Petersburg,
in which I was invited to take part,
I made the acquaintance with
his wife Ute and daughter Kim;
two years later I met his other daughter Tanja.
Heiner guided my study of coincidences,
intersections and preimages 
during my visit
in November--December 2004 in
Ruhr-Universit\"at Bochum.
It was planned, to continue the project in~2005.
But that hope was doomed to disappointment\dots.}

\volumenumber{14}
\issuenumber{}
\publicationyear{2008}
\papernumber{011}
\startpage{193}
\endpage{217}

\doi{}
\MR{}
\Zbl{}

\arxivreference{} 


\received{30 March 2006}
\revised{4 March 2007}
\accepted{18 April 2007}
\published{29 April 2008}
\publishedonline{29 April 2008}
\proposed{}
\seconded{}
\corresponding{}
\version{}

\AtBeginDocument{\let\bar\wbar\let\tilde\wtilde\let\hat\what}
\AtBeginDocument{\renewcommand\R{\operatorname{R}}}
\AtBeginDocument{\def\S{Section }}
\def\smallsim{\hbox{\footnotesize{$\sim$}}}

\newtheorem{predl}{Proposition}
\newtheorem{theorem}{Theorem}
\newtheorem{cor}{Corollary}
\newtheorem{lemma}{Lemma}

\theoremstyle{definition}
\newtheorem{defin}{Definition}
\newtheorem{rem}{Remark}

\makeautorefname{predl}{Proposition}

\def\Coin{\operatorname{Coin}}
\def\id{\operatorname{id}}

\def\NT{\operatorname{N}_{\text{\rm t}}}
\def\MP{\operatorname{MP}}

\def\RR{\operatorname{\mathcal R}}

\def\MPCL{\operatorname{MP}_{\text{\rm cl}}}

\begin{document}

\begin{asciiabstract}
We find conditions on topological spaces X, Y and nonempty subset B of
Y which guarantee that for each continuous map f from X to Y there
exists a map g homotopic to f such that Nielsen preimage classes of
g^{-1}(B) are all topologically essential.
\end{asciiabstract}
$\phantom{9}$   
\begin{htmlabstract}
We find conditions on topological spaces X, Y and nonempty subset
B of Y which guarantee that for each continuous map
f:X&rarr; Y there exists a map g&sim; f such that
Nielsen preimage classes of g<sup>-1</sup>(B) are all topologically
essential.
\end{htmlabstract}

\begin{abstract}
We find conditions on topological spaces $X$, $Y$ and nonempty subset
$B$ of $Y$ which guarantee that for each continuous map
$f\colon\thinspace X\to Y$ there exists a map $g\sim f$ such that
Nielsen preimage classes of $g^{-1}(B)$ are all topologically
essential.
\end{abstract}

\maketitle

\section{Introduction}\label{Section 1}

Let $f\co X \to Y$ be a continuous map of topological spaces
and $B$ be a nonempty subsets of~$Y$.
The so-called preimage problem considers the
preimage set $f^{-1}(B)$, that is, $\{ x\in X | f(x)\in B\}$.
The minimization problem 
in its ``classical'' setting
is to compute 
or at least find a good (lower) estimate for
the number
$$\MP (f,B) = \min\limits_{g\smallsim f} | g^{-1}(B) | ,$$
where $|g^{-1}(B)|$ is the cardinality of the set $g^{-1}(B)$
and minimum is taken over all maps $g$ homotopic to $f$.

This problem was considered in detail 
by Dobre{\'n}ko and Kucharski \cite{dk} (see also Schir\-mer \cite{Koinz} and
Jezierski \cite{preim-top}); if $B$ is a point, the problem is called
the root problem and dates back to Hopf \cite{hopf1,hopf2}.
Other problems of this type are, for example, fixed point, coincidence and 
intersection problems; see Jiang \cite{jiang}, Bogaty{\u\i}, Gon\c{c}alves and Zieschang \cite{bgz} and McCord \cite{int}
(\cite{mccord} for their interrelations).
All these problems can be attacked by a Nielsen-type technique.
So, in order to estimate the number $\MP (f,B)$,
the preimage set $f^{-1}(B)$ is divided into equivalence
classes bearing the name of Nielsen. 
The Nielsen number $\NT (f,B)$ is the number of so-called
(topologically; therefore we put a 
straight letter ``t'' in the notation,
which is not to be mixed up with a homotopy or isotopy parameter) 
essential classes. 
It is a homotopy invariant, and 
$\NT (f,B) \leqslant \MP (f,B)$ \cite[Theorem (1.9)]{dk}.
If $\NT (f,B) = \MP (f,B)$, it is said
that the setting $f\co X\to Y\supset B$  
has the Wecken property, or
a Wecken type theorem holds true.
Assuming $\NT (f, B) < \infty $ 
(see eg \fullref{Nielsen-finite} 
and \fullref{Nielsen-finite-DK}),
the Wecken property holds iff
there exists a map $g\sim f$ which has exactly
$\NT (f,B)$ Nielsen classes, moreover, each of them contains
only one point.

For $X$, $Y$, $B$ 
manifolds with $\dim X= \dim Y-\dim B\geqslant 3$
Wecken type theorems hold true; see Dobre{\'n}ko and Kucharski
\cite[Theorem~(3.4)]{dk}, Jezierski \cite[Theorem~(3.2)]{preim-top} and 
Frolkina \cite{fro2,fro1}
for maps of pairs of smooth manifolds.
The case of a surface $X$ is more complicated. 
If $Y$ is also a surface and $B$ is a finite set,
this problem is solved; see Bogaty{\u\i}, Gon\c{c}alves,  Kudryavtseva and Zieschang \cite[p~17, Remark~(e)]{bgkz}.
But, as it is noted by McCord \cite[p~175]{mccord},
for each surface~$S$ of negative Euler characteristic
reasoning of Jiang \cite{jiang-i,jiang-ii} 
provides examples of preimage problems of the form
$f\Delta \id_S \co S\to S\times S\supset \Delta S$
($\Delta S$ is the diagonal) such that
$
1=\NT (f\Delta \id _S, \Delta S) <
\MP (f\Delta \id _S, \Delta S) =2
$.
If $\dim X > \dim Y - \dim B$,
the number $\MP (f,B)$ is ``usually'' infinite;
we could consider instead the minimal number 
of path components of the preimage set, 
as is done by Koschorke \cite{kos1} for coincidences.
We will not deal with such a problem here.

For general topological spaces, obtaining Wecken type theorems seems 
to be a very complicated problem whose solution 
depends on concrete spaces and maps; see 
Brooks \cite[p~102]{delta}, \cite[Example~(3.15)]{roots-review} for the
root problem.
In this paper
we will consider (for preimage case)
the following question:
does there exist a map $g\sim f$
which has exactly $\NT(f,B)$ Nielsen classes
(not necessary consisting of one point each).
That is, defining
$$
\MPCL (f,B) = \min\limits_{g\smallsim f}
|\{ \text{Nielsen classes of } g^{-1}(B) \} | ,
$$
we have 
$\NT(f,B)\leqslant 
\MPCL (f,B)\leqslant \MP(f,B)$, and the question
concerns exactness of the left inequality.
But we will not restrict ourselves to the 
case $\NT (f,B) < \infty $.
Therefore we extend the posed question as follows:
does
there exist a map $g\sim f$, whose
Nielsen classes are all topologically essential.
This is not always possible; a counterexample
can be obtained converting
(see McCord \cite{mccord} 
or the end of \fullref{Section 1})
the coincidence problem of
\cite[Example~2.4]{gw} 
into the preimage problem.
For coincidences and roots, 
this problem was introduced by Brooks
\cite{delta}, who stated sufficient 
conditions for a positive answer.
(Refer also
to the paper Gon\c{c}alves and Aniz \cite{ga} 
devoted to the question of simultaneous
minimization of a number of points in all root classes.)
Our main theorem (\fullref{minim}) unites and generalizes the results 
\cite[Theorems~1,2]{delta}. 
We also discuss in detail Nielsen classes 
and Nielsen number.

Before starting, we would like to underline the following.
The classical Nielsen fixed point number
(the idea of its definition belongs to J\,Nielsen \cite{Nielsen})
is the number of \textit{algebraically} essential
fixed point classes, that is, classes of nonzero index.
Local fixed point index is defined
for maps of compact metric ANRs; but
there does not seem to be a well-developed index theory
for the root problem, except in, for example, the manifold case
(see Brooks
\cite[pp~376, 381--382, section~4]{roots-review} and
Gon\c{c}alves \cite[p~24]{gonc-coinc}).
Therefore Brooks prefers to use the notion of \textit{topological} 
essentiality which does not depend on the existence of 
local root index. Since the root problem is a particular case
of the preimage problem, by the same reason we prefer to consider
topological essentiality rather than algebraic;
the corresponding ``topological'' Nielsen number was defined
in Dobre{\'n}ko and Kucharski \cite{dk}. (Note that for $X$, $Y$, $B$ manifolds with
$\dim X = \dim Y - \dim B \geqslant 3$ the two notions of
essentiality coincide \mbox{\cite{dk,preim-top}}.) 
It is defined for arbitrary spaces; in
particular, this allows us to omit compactness assumptions.

\subsection*{Conventions and notation}

Throughout this paper 
spaces $X$, $Y$ 
are Hausdorff, connected, locally
path connected; moreover,
$Y$ is semilocally simply connected;
$B$ is a nonempty subset of $Y$;
the same is suggested for $X^\prime $, $Y^\prime $, $B^\prime $.
In our main statements we will additionally repeat this
and, if necessary, require something else.

For topology of infinite polyhedra, the 
reader should refer to Spanier \cite[Chapter~3]{sp};
when we make use of a concrete theorem,
we will give a more detailed reference.
For a polyhedron~$X$ 
we denote by~$X^{(n)}$ its $n$--skeleton;
$I$ is the unit segment $[0,1]$.

As usual, all covering spaces are (assumed to be) connected.

All maps are assumed to be continuous.
By $\id _X$ we denote the identity map of a space $X$;
$fg$ is the composition of maps~$f$ and~$g$;
$\Delta \{ f_i \}$ is the diagonal product of the 
family of maps $\{ f_i \}$;
for $A\subset X$ and an integer $r$ the symbol $\Delta A \subset X^r$
is used for the image of $A$ under 
the diagonal product of $r$ embeddings
$A\hookrightarrow X$.
For a subset $A\subset X\times I$ by its $t$--section, where
$t\in I$, we mean the set $A \cap X \times \{ t\}$.
Speaking about homeomorphisms and homotopy equivalences
of triples of spaces, 
we mean of course morphisms of the appropriate category.

For a homotopy $\{ f_t\} \co X\to Y$ and a path $\alpha \co I\to X$,
by $\{ f_t \alpha(t) \}$ we denote clearly a path 
$F (\alpha \Delta \id _{I})$
in $Y$, where 
$F\co X\times I\to Y$ is given by $(x,t)\mapsto f_t(x)$.

The symbol $\sim $ means homotopy of maps and
homotopy of paths relative to end points; 
$[\alpha ]$ is the homotopy class of a path $\alpha $
(again relative to end points); 
by $\alpha \cdot \beta $ 
we denote the product of paths $\alpha $ and $\beta $
with $\alpha (1) = \beta (0)$ and by 
$[\alpha ]\cdot [\beta ]$ the product of their homotopy classes.

For a map $f\co (X,x_0)\to (Y,y_0)$ we denote by
$f_{\# }\co \pi _1(X,x_0)\to \pi _1(Y,y_0)$ the induced
homomorphism.
If in the notation of (relative) homotopy
groups $\pi_m (X,A,x_0)$ (also for the set $\pi_1 (X,A,x_0)$) 
we omit the base point, 
we have in mind that $A$ is (suggested to be) path connected.

We use singular (co)homology with coefficients in local systems
of groups; if no coefficients are designated, 
they are usual (``constant'') integers.

Other notation is either standard or is introduced in the
text.

\subsection*{Statement of the main theorem}

\begin{theorem}\label{minim}
Suppose that
the spaces $X$, $Y$ are
connected and locally path connected;
moreover, $Y$ is semilocally simply connected, and
$B$ is a nonempty locally path connected closed subspace.
Suppose that for some integer $n\geqslant 3$
the space $X$ is dominated by 
a polyhedron of dimension less or equal to~$n$
and $\pi_m(Y,Y-B)=0$
for all $1\leqslant m\leqslant n-1$.
Then for each map $f\co X\to Y$ 
there exists a map $g\sim f$ such that
each Nielsen preimage class of $g\co X\to Y\supset B$ is
topologically essential; in particular,
$\NT (f,B) = \MPCL (f,B)$.
\end{theorem}

This means that under the above conditions
we can delete all inessential preimage classes of $f$
\textit{at once} (recall that
each single inessential class can be deleted
by definition; see \fullref{top-ess} below).

\begin{rem}\label{pi1}
For a path connected space $Y$ and a subspace $B$,
${\pi _1(Y,Y-B)=0}$ 
if and only if 
$B$ can be bypassed in~$Y$;
see Schirmer \cite[Theorem~5.2]{rel-f}.
Recall from \cite[Definition~5.1]{rel-f} that
a subspace~$B\subset Y$ \textit{can be bypassed in $Y$}
if every path in $Y$ with end points in~$Y-B$ 
is homotopic to a path in~$Y-B$.
Recall also that
for a triangulated manifold~$Y$,
its boundary $\partial Y$
or a subpolyhedron $B$ with $\dim Y - \dim B \geqslant 2$
provide examples of subspaces which
can be bypassed \cite[p~468]{rel-f}.
\end{rem}

\begin{rem}
As in Brooks \cite[p~382--383]{roots-review}, 
recall the following well-known results.
If an $n$--dimensional paracompact space is dominated by a 
polyhedron, then it is dominated by a polyhedron of
dimension~$n$ or less; see Granas and Dugundji \cite[Theorem 17.7.16(c), p~483]{gd}.
Each ANR is dominated by a polyhedron 
by Hu \cite[Chapter~I, Exercise~R, p~32]{hu}
(see also Borsuk \cite[Corollary~(V.4.5)]{borsuk}).
\end{rem}

\begin{rem}
For arbitrary compact space~$X$,
other partial results of Wecken type (for roots)
can be found in Gon\c{c}alves and Wong~\cite{gw}.
\end{rem}

It is clear that the root problem is a particular case
of the preimage problem. But the
coincidence problem of $f,g\co X\to Y$ 
is also equivalent to the preimage problem
$f\Delta g\co X\to Y\times Y\supset \Delta Y$
(for details, see \fullref{many} of the present paper
and references there). 
Therefore the results \cite[Theorems~1,~2]{delta} 
can be derived from our \fullref{minim}. 

In order to prove \fullref{minim},
we firstly define and investigate 
Nielsen classes (see \fullref{pre-classes}). 
We will need special properties of Nielsen number
(see \fullref{triangle}) which
will be used also to prove homotopy
invariance (on spaces) of Nielsen number
(see \fullref{property}). 

\section{Nielsen classes}

Preimage points $f^{-1}(B)$ are divided into
so-called Nielsen preimage classes (we also
say simply ``preimage classes'' or ``Nielsen classes'';
sometimes we speak about classes of 
the problem $f\co X\to Y\supset B$ or of the map $f$).

\begin{defin}{\rm \cite[Definition~(1.2)]{dk}}\qua
Points
$x_0, x_1\in f^{-1}(B)$ are said to be
(Nielsen) equivalent if there are paths
$$\alpha \co (I,0,1)\to (X,x_0,x_1) \quad\text{and}\quad 
\beta \co (I,0,1)\to (Y,f(x_0),f(x_1))$$ 
such that $\beta (I)\subset B$ and $f \alpha \sim \beta $ 
(homotopy in~$Y$).
\end{defin}

Let $\{f_t \} \co X\to Y$ be a homotopy. 

\begin{defin}
A point $x_0\in f_0^{-1}(B)$ is said to be
$\{f_t\}$--related 
to a point $x_1\in f_1^{-1}(B)$
if there exist paths $\alpha \co (I,0,1)\to (X,x_0,x_1)$ 
and
$\beta \co (I,0,1)\to (Y,f_0(x_0), f_1(x_1))$ 
such that
$\beta (I) \subset B$
and
$\{ f_t \alpha (t) \} \sim \beta $ (homotopy in $Y$). 
\end{defin}

(Another equivalent definition of
the $\{ f_t\}$--relation is \cite[Definition~(1.6)]{dk};
see also Item (5) of \fullref{pre-classes} below.)

Note that two preimage points of $f\co X\to Y\supset B$
are Nielsen equivalent iff 
they are related by a constant homotopy $\{ f_t = f \}$.

It is clear that 
the following definition makes sense:

\begin{defin}
A preimage class
$A_0\subset f_0^{-1}(B)$ is $\{f_t\}$--related to a preimage 
class $A_1\subset f_1^{-1}(B)$ if at least one (and hence
every) preimage point $x_0\in A_0$ is
$\{f_t\}$--related to at least one (and therefore every) 
preimage point $x_1\in A_1$.
\end{defin}

Note that each preimage
class of $f_0$ is $\{ f_t\}$--related to at most one 
preimage class of $f_1$ but may not be 
$\{ f_t\}$--related to any
class of $f_1$; see \fullref{top-ess} below.

A useful tool in connection with Nielsen classes is the
following:

\begin{defin}
A Hopf covering and a Hopf lift 
for a map $f\co X\to Y$ are, respectively,
a covering $p\co \hat Y\to Y$ and 
a lift $\hat f\co X\to \hat Y$ of $f$ such that
$p_{\# }(\pi_1(\hat Y,\hat f (x))) = f_{\# }(\pi_1(X,x))$
for each $x\in X$.
\end{defin}

We use here the name of Hopf following Brooks \cite{roots-review},
since Hopf was the first who used such covering and lifts
in Nielsen root theory \cite{hopf2}.

If it is desirable to underline that the covering
$(\hat Y ,p)$ depends on a given map $f$, we will
write $(\hat Y_f , p_f)$ or $(\hat Y, p_f)$.

The following remarks (see Hopf \cite[\S 2]{hopf2} and
Brooks \cite[p~379]{roots-review}) will be used below.
Recall that we assume the conventions of \fullref{Section 1}.

\begin{predl}\label{Hopf-c-l}
\begin{itemize}
\item[(1)] 
Hopf coverings and lifts always exist.
\item[(2)] 
A Hopf covering is unique up to covering space isomorphism.
\item[(3)] 
For 
a Hopf covering $p_f\co \hat Y_f\to Y$ and
two Hopf lifts $\hat f^{(1)}$, $\hat f^{(2)}$ of $f$,
there exists a covering transformation $\rho \co \hat Y_f\to \hat Y_f$ 
such that
$\hat f^{(1)} = \rho \hat f^{(2)}$.
\item[(4)] 
If $(\hat Y_f,p_f)$ is a Hopf covering for $f$, then it is 
a Hopf covering for each map $f_1\sim f$.
\item[(5)] 
If $\hat f$ is a Hopf lift for $f$ and $\{ \hat f_t\}$ is a lift
of a homotopy $\{ f_t\} \co f_0=f\sim f_1$ 
such that $\hat f_0 = \hat f$,
then $\hat f_1$ is a Hopf lift for $f_1$.
\end{itemize}
\end{predl}

\begin{proof}
(1)\qua
Fix an arbitrary point $x_0\in X$. 
Take a covering $p\co \hat Y\to Y$ that corresponds to the
subgroup $f_{\# }(\pi _1(X,x_0))\subset \pi_1(Y,f(x_0))$.
That is, for some point $\hat y_0\in p^{-1}(f(x_0))$
we have $p_{\# }(\pi_1 (\hat Y,\hat y_0))=f_{\# }(\pi _1(X,x_0))$
\cite[Theorem~2.3.6]{sp}.
Then by \cite[Theorem~2.4.5]{sp} there exists a lift
$\hat f\co X\to \hat Y$ of $f$ such that
$\hat f(x_0) = \hat y_0$. 
(Note that a lift $\hat f$ is not necessary uniquely determined, 
because of possible ambiguity in choice of $\hat y_0$.)
So, $p_{\# }(\pi_1 (\hat Y,\hat f(x_0)))=f_{\# }(\pi _1(X,x_0))$.
From path connectivity of~$X$ it follows
that the equality
$p_{\# }(\pi_1 (\hat Y,\hat f(x)))=f_{\# }(\pi _1(X,x))$
holds for each point $x\in X$.
Indeed, take a path $\gamma \co (I,0,1)\to (X,x,x_0)$; we have
\begin{align*}
p_{\# }(\pi_1 (\hat Y,\hat f(x)))&=
p_{\# }\left( [\hat f(\gamma )]  \cdot
\pi_1 (\hat Y,\hat f(x_0)) \cdot 
[\hat f (\gamma ^{-1})] \right)
\\
&=
[p_{\# }\hat f(\gamma )]  \cdot
p_{\# }(\pi_1 (\hat Y,\hat f(x_0))) \cdot
[p_{\# }\hat f (\gamma ^{-1}) ]
\\
&=
[f(\gamma )] \cdot
f_{\# }(\pi _1(X,x_0)) \cdot
[f (\gamma ^{-1})] =
f_{\# }(\pi _1(X,x)).
\end{align*}
(2)\qua This is clear; see eg Massey \cite[Corollary~V.6.4]{massey}.

(3)\qua Take an arbitrary $x_0\in X$. We have
$$
(p_f)_{\# } (\pi_1(\hat Y, \hat f^{(1)}(x_0))) =
(p_f)_{\# } (\pi_1(\hat Y, \hat f^{(2)}(x_0))).
$$
Hence \cite[Corollary~V.6.4]{massey} 
there exists a covering transformation $\rho $
of the covering $\smash{p_f\co \hat Y_f\to Y}$ such that
$\smash{\hat f^{(1)}(x_0) = \rho \hat f^{(2)}(x_0)}$.
Therefore 
$\smash{\hat f^{(1)} = \rho \hat f^{(2)}}$ \cite[Theorem~2.2.2]{sp}.

We prove (4) and (5) simultaneously, using notation
common for these two items: $(\hat Y_f,p_f)$ is a Hopf covering
for $f$, $\{ f_t\} \co f_0 = f \sim f_1$ a homotopy.
Take an arbitrary point $x_0\in X$.
According to \cite[Theorem~1.8.7]{sp}, we have
$$
(f_1)_{\# } (\pi _1(X,x_0)) = 
[\omega ] \cdot f_{\# } (\pi _1(X,x_0)) \cdot [\omega ^{-1}],
$$
where 
$\omega \co (I,0,1) \to (X, f_1(x_0), f(x_0))$ is defined by
$\omega (t) = f_{1-t} (x_0)$. 
Then, putting
$\hat\omega (t) = \hat f_{1-t} (x_0)$, we obtain
\begin{align*}
(p_f)_{\# } (\pi _1(\hat Y, \hat f_1(x_0))) &= 
(p_f)_{\# } ([\hat\omega ]\cdot \pi _1(\hat Y, \hat f(x_0)))
\cdot [\hat\omega ^{-1}] ) 
\\
&=
[\omega ]\cdot (p_f)_{\# }(\pi _1(\hat Y,\hat f(x_0))) \cdot
[\omega ^{-1}] \\&=
[\omega ]\cdot f_{\# }(\pi _1(X,x_0)) \cdot
[\omega ^{-1}] = (f_1)_{\# }(\pi_1(X,x_0)).
\end{align*}
The necessary statement follows now as in the proof of~(1).
\end{proof}

Returning to preimage classes, we obtain the
following description of Nielsen classes
and $\{ f_t\}$--relation 
(see Hopf \cite[Satz~III]{hopf2} and
Brooks \cite[Theorem~(3.4)]{roots-review} for roots):

\begin{theorem}\label{pre-classes}
Let $(\hat Y,p)$ and
$\hat f$ be a Hopf covering and a Hopf lift
for ${f\co X\to Y\supset B}$.
Let $\{ f_t \} \co X\to Y$ be a homotopy from $f_0=f$ to $f_1$
and $\{ \hat f_t \} \co X\to \hat Y$ its lift 
such that $\hat f_0 = \hat f$. 
Then
\begin{itemize}
\item[(1)]
two preimage points $x_0,x_1\in f^{-1}(B)$ are Nielsen
equivalent if and only if
the points $\smash{\hat f(x_0)}$, $\smash{\hat f(x_1)}$ lie in the
same path component of the set $p^{-1}(B)$;
\item[(2)]
Nielsen classes of 
$f\co X\to Y\supset B$ are precisely 
nonempty sets of the form ${\hat f} ^{-1}(C)$, where~$C$
is a path component of the set $p^{-1}(B)$;
\item[(3)]
a point $x_0\in f_0^{-1}(B)$ is $\{f_t\}$--related to a
point $x_1\in f_1^{-1}(B)$ if and only if 
the points $\smash{\hat f_0 (x_0)}$, $\smash{\hat f_1 (x_1)}$ 
are contained in the same path component
of the set $p^{-1}(B)$;
\item[(4)]
a preimage class
$A_0\subset f_0^{-1}(B)$ is $\{f_t\}$--related to a 
class $A_1\subset f_1^{-1}(B)$ if and only if
the sets
$\smash{\hat f_0 (A_0)}$ and $\smash{\hat f_1 (A_1)}$ are contained
in one path component of the set $p^{-1}(B)$;
\item[(5)] 
a preimage class
$A_0\subset f_0^{-1}(B)$ is
$\{f_t\}$--related to a class
$A_1\subset f_1^{-1}(B)$ 
if and only if $A_0$, $A_1$ are
$0$-- and $1$--sections of some preimage class
of 
$F\co X\times I \to Y\supset B$, where $F(x,t)=f_t(x)$.
\end{itemize}
\end{theorem}

\begin{proof}
It is clear that $(3)\Rightarrow (1)\Rightarrow (2)$ 
and
$(3)\Rightarrow (4)\Rightarrow (5)$.
Let us prove~(3).

Suppose the points $x_0$, $x_1$ are $\{f_t\}$--related, that is,
for some paths
$\alpha \co (I,0,1)\to (X,x_0,x_1)$ and
$\beta \co (I,0,1)\to (B,f_0(x_0),f_1(x_1))$
we have
$[ \{ f_t  \alpha (t) \} ] = [ \beta ] $.
Let $\hat \beta \co (I,0,1) \to 
(\hat Y, \hat \beta (0)=\hat f_0 (x_0), \hat \beta (1))$ 
be the lift of $\beta $ into $\hat Y$ beginning 
at $\hat f_0 (x_0)$. 
Denote by $C$ the path component of
$p^{-1}(B)$ that contains $\hat f_0(x_0)$; 
we have $\beta (I)\subset C$.
The path $\{ \hat f_t  \alpha (t) \}$ is homotopic to $\hat \beta $
and starts at the same point $\hat f_0(x_0)$. 
Hence $\hat \beta (1) = \hat f_1(x_1)$, what implies 
$\hat f_1(x_1)\in C$.

To prove the converse, suppose
that the points $\hat f_0(x_0)$, $\hat f_1(x_1)$ lie in the 
same path component $C$ of $p^{-1}(B)$.
Consider arbitrary paths
$\alpha \co (I,0,1)\to (X,x_0,x_1)$
and
$\hat\beta \co (I,0,1)\to (\hat Y, \hat f_0(x_0), \hat f_1(x_1))$
with $\hat \beta (I) \subset C$.
Denoting $\beta = p\hat \beta $, we have $\beta (I) \subset B$.
Then $\{ \hat f_t \alpha (t)\} \cdot \hat\beta ^{-1}$ is a loop
at $\hat f_0(x_0)$.
Therefore the path 
$p (\{ \hat f_t  \alpha (t) \} \cdot \hat\beta ^{-1})$
is a loop at $f(x_0)$. By definition of Hopf covering,
there exists a loop $\gamma $ at $x_0$ in $X$ such that
$$
[ f\gamma ] = 
[p (\{ \hat f_t \alpha (t) \}\cdot \hat\beta ^{-1})] =
[ \{ p \hat f_t \alpha (t) \} \cdot
\beta ^{-1} ] = [\{ f_t \alpha (t) \} ] \cdot
[\beta ^{-1}].
$$
Consequently,
$$
[\{ f_t ((\gamma^{-1}\cdot \alpha )(t)) \} ] =
[f \gamma ]^{-1} \cdot [ \{ f_t \alpha (t) \} ] =
[\beta ].
$$
The last equality shows that $x_0$ 
is $\{ f_t \}$--related to~$x_1$.
\end{proof}

Note that (5) of \fullref{pre-classes} 
is taken in \cite[Definition~(1.6)]{dk} as a definition 
(the $\{ f_t\}$--relation is called there ``$F$--Nielsen relation'').

From this theorem, we derive a number of
simple corollaries.
The first one is evident;
it generalizes \cite[Theorem~(3.8)]{roots-review}
(see \fullref{top-ess} 
for the notion of topological essentiality):

\begin{cor}\label{essent}
Let $f\co X\to Y\supset B$ be a map, 
$(\hat Y_f,p_f)$ and $\hat f\co X\to \hat Y$
its Hopf covering and lift and
$\hat B$ a path component of $p^{-1}(B)$. 
Then
$\smash{\hat f ^{-1}(\hat B)}$ is
a
topologically essential 
preimage class 
if and only if
$\smash{\hat f_1 ^{-1} (\hat B)} \neq \emptyset $ for any
homotopy $\smash{\{\hat f_t\}}$ beginning at $\smash{\hat f_0=\hat f}$.
\end{cor}

Let us underline once more that in general
the Nielsen number $\NT (f,B)$ may be infinite.
We obtain now simple sufficient conditions for
its finiteness;
see Hopf \cite[Satz~II, IIa]{hopf2} and Brooks
\cite[Theorem~(3.5), Corollary~(3.6)]{roots-review} for roots.

\begin{cor}\label{clopen}
Suppose (additionally to our usual conventions) 
that $B$ is locally path connected. 
Then
each Nielsen class of $f\co X\to Y\supset B$ 
is both open and closed in $f^{-1}(B)$.
Hence, if $f^{-1}(B)$ is compact, 
then the number of Nielsen classes is finite.
\end{cor}

\begin{proof}
The second statement is clear.
To prove the first,
it suffices to show that each Nielsen class is an
open subset of the preimage set.
Let $x\in f^{-1}(B)$.
Take an open neighbourhood $V_1\subset Y$ of $f(x)$
such that every two paths in $V_1$ 
starting at $f(x)$ and
having the same end points
are homotopic in $Y$. 
Let $W\subset B\cap V_1$ be an open (in $B$) 
path connected neighbourhood of $f(x)$. 
Then $W=B \cap V_2$ for some open set $V_2\subset Y$.
Put $V=V_1\cap V_2$. Take an open 
path connected neighbourhood $U\subset X$
of $x$ such that $f(U)\subset V$. Suppose $y\in U\cap f^{-1}(B)$.
Take paths $\alpha \co (I,0,1)\to (U,x,y)$
and $\beta \co (I,0,1)\to (W,f(x),f(y))$.
Then $f\alpha \co (I,0,1)\to (V,f(x),f(y))$ is
homotopic (in $Y$) to $\beta $; that is, the points $x$, $y$
belong to the same preimage class.
\end{proof}

\begin{cor}\label{Nielsen-finite}
Suppose (additionally to our usual conventions) 
that $B$ is locally path connected and
at least one of the following conditions holds:
\begin{itemize}
\item[(1)]
$X$ is compact, 
and $B$ is closed in $Y$;
\item[(2)]
$f$ is proper, and $B$ is compact.
\end{itemize}
Then the number of Nielsen classes is finite.
\end{cor}

\begin{proof}
In both cases, $f^{-1}(B)$ is compact. 
Application of \fullref{clopen} 
finishes the proof.
\end{proof}

\begin{rem}\label{Nielsen-finite-DK}
See the statement and proof of \cite[Theorem~(1.3)]{dk}
for a different list of conditions on spaces $X$, $Y$, $B$ which
also imply 
openness of Nielsen classes and finiteness of its number.
\end{rem}

\section[R-sets]{$\RR $--sets}\label{R-sets}

Take $f\co X\to Y\supset B$, its Hopf covering $(\hat Y,p)$, and 
a Hopf lift $\hat f$.
For each path component $\smash{\hat B}$ of $\smash{p^{-1}(B)}$, 
call its preimage $\smash{\hat f^{-1}(\hat B)}$ an $\RR $--set.
Some
$\RR $--sets may be empty, but we nevertheless
distinguish them through the path components $\hat B$
which define them.
That is, an $\RR $--set 
$\hat f ^{-1}(\hat B)$
is considered to carry a label $\hat B$.
The collection of all $\RR $--sets for a map $f$ and its
Hopf lift $\hat f$ is denoted by $\RR (f, \hat f)$.

Note (see Brown and Schirmer \cite[Remark~4.5]{rel-roots}) that the number of path
components of $p^{-1}(B)$ and hence of $\RR $--sets
equals the Reidemeister preimage number 
$\R (f,B)$, which is often useful for computation of 
the Nielsen number (see \fullref{top-ess}), and
this is the reason of the presence of the letter $\RR $
in the name of $\RR $--sets. 
We will not go into details,
but refer to Brooks \cite{roots-review}, for example, for roots.

From \fullref{pre-classes} the following proposition holds:

\begin{predl}
Nonempty $\RR $--sets (considered without labels) 
are exactly Nielsen classes.
\end{predl}

We have already defined $\{f_t\}$--relation between Nielsen
classes. Now we will extend this relation to $\RR $--sets.

Let $\{ f_t\} \co f_0\sim f_1$ be a homotopy.
By \fullref{Hopf-c-l}, a Hopf covering
$(\hat Y,p)$ constructed for $f_0$ 
is also a Hopf covering for $f_1$. 
Further, if $\hat f_0$ is a Hopf lift for $f_0$, 
and $\{ \hat f_t \}$ is a lift of the homotopy $\{ f_t \}$
starting at this $\smash{\hat f_0}$, 
then $\smash{\hat f_1}$ is a Hopf lift for $f_1$. 
The homotopy $\{f_t\}$ induces therefore a bijection between
$\RR $--sets of $f_0$ and $f_1$ by the following rule:
$$
\RR (f_0, \hat f_0) \leftrightarrow
\RR (f_1, \hat f_1), 
\quad
\hat f_0 ^{-1}(\hat B) \leftrightarrow \hat f_1 ^{-1}(\hat B),
$$
for each path component $\hat B$ of $p^{-1}(B)$.
From \fullref{pre-classes} it follows that
for Nielsen classes (equivalently, nonempty $\RR $--sets
considered without labels)
this is just an $\{ f_t \}$--relation.
Hence it gives a bijection between the sets 
of all topologically essential
preimage classes of $f_0$ and $f_1$
(see \fullref{top-ess}).

We will need the following lemma.

\begin{lemma}\label{q}
Suppose that the diagram
$$
\begin{CD}
X @> f >> Y
\\
@V \varphi VV @| 
\\
X^\prime @> f^\prime >> Y 
\end{CD}
$$
commutes; let 
$(\hat Y_{f}, p_{f})$,
$(\hat Y_{f^\prime }, p_{f^\prime })$ and
$\hat f$, $\hat f^\prime $
be Hopf coverings and Hopf lifts for $f$, $f^\prime $.
Then
\begin{itemize}
\item[(1)]
there exists a unique covering 
$q\co \hat Y_{f} \to \hat Y_{f^\prime }$
such that
$q \hat f = \hat f^\prime \varphi $
and 
$p_f = p_{f^\prime } q$;
\item[(2)]
a map
$\RR (f , \hat f)
\to \RR (f^\prime , \hat f^\prime ) $,
$\hat f ^{-1} (\hat B ) \mapsto
(\hat f^\prime )^{-1} (q \hat B ) $,
where $\hat B \subset \hat Y_f $ is a path component 
of $p_f ^{-1}(B)$, is well-defined;
this map brings a nonempty $\RR $--set 
(considered without label), that is,
a preimage class $A$ of the problem 
$f\co X\to Y\supset B$,
to that (nonempty) $\RR $--set, that is, preimage
class $A^\prime $ of $f^\prime \co X^\prime \to Y \supset B$,
which contains $\varphi (A)$.
\end{itemize}
\end{lemma}

\begin{proof}
(1)\qua Take an arbitrary point $x_0\in X$.
We have 
\begin{multline*}
(p_f)_{\# } (\pi _1(\hat Y_f , \hat f(x_0))) =
f_{\# } (\pi _1(X,x_0)) 
=
(f^\prime \varphi )_{\# } (\pi _1(X,x_0)) 
\\
\subset
f^\prime _{\# }(\pi _1(X^\prime , \varphi (x_0)))
= (p_{f^\prime })_{\# } 
(\pi _1 (\hat Y_{f^\prime }, \hat f^\prime \varphi (x_0))).
\end{multline*}
Hence there exists a unique covering 
$\smash{q \co (\hat Y_f , \hat f (x_0))\to 
(\hat Y_{f^\prime }, \hat f^\prime \varphi (x_0))}$
such that $p_f = p_{f^\prime } q$.
To prove the equality
$\smash{q \hat f = \hat f^\prime \varphi} $,
note that the two maps 
$\smash{q\hat f, \hat f^\prime \varphi \co X\to \hat Y_{f^\prime }}$
are lifts of the same map $f$ and coincide on $x_0$.

(2) follows from
$\varphi (\hat f ^{-1} \hat B)
\subset (\hat f^\prime )^{-1} (q\hat B)$.
\end{proof}

\section{Nielsen number}

\begin{defin}{\rm \cite[Definition (1.8)]{dk}}\qua\label{top-ess}
A preimage class $A_0$ of $f\co X\to Y\supset B$
is called topologically essential if
for each homotopy $\{ f_t \}\co X\to Y$ beginning at $f_0=f$
there is a Nielsen class $A_1$ of
$f_1\co X\to Y\supset B$ which is $\{f_t\}$--related to $A_0$;
that is, the class $A_0$ can not ``disappear'' under
homotopies, or there is no homotopy which can ``delete'' the
class~$A_0$. Otherwise the class $A_0$ is called inessential.
The number of topologically essential preimage classes
is called the topological Nielsen number
of the given preimage problem, or of the map
$f$ with respect to $B$, and it is denoted by
$\NT (f\co X\to Y\supset B)$ or shortly $\NT (f,B)$;
it is an integer or infinity.
\end{defin}

In the present paper we often omit the word ``topologically'',
since we do not consider the other 
type of essentiality (algebraic).
Note that some authors 
use the word ``geometrical'' instead of ``topological''.

It follows from the definition of essential class
(see also \cite[Theorem~(1.9)]{dk}) that:

\begin{predl}
The Nielsen number $\NT (f,B)$ is a homotopy 
invariant of a map~$f$ and
$\NT (f,B) \leqslant \MP (f,B)$.
\end{predl}

The next theorem implies stronger invariance of the
Nielsen number.

\begin{theorem}\label{property}
Suppose that the diagram
$$
\begin{CD}
X @> f >> (Y,B,Y-B)
\\
@V \varphi VV @V \psi VV
\\
X^\prime @> f^\prime >>
(Y^\prime ,B^\prime , Y^\prime - B^\prime )
\end{CD}
$$
commutes up to homotopy and $\psi $ is a
homotopy equivalence.
If
$\varphi _{\#} \co \pi_1(X ,x_0) 
\to \pi_1(X^\prime ,\varphi (x_0))$ is surjective,
then 
${\NT (f\co X\to Y \supset B)\leqslant 
\NT(f^\prime \co X^\prime \to Y^\prime \supset B^\prime )}$.
If moreover $\varphi $ has a right homotopy inverse,
then 
$\NT (f\co X\to Y \supset B) =
\NT(f^\prime \co X^\prime \to Y^\prime \supset B^\prime )$.
\end{theorem}

To prove this, we need two lemmas.
The first unites and generalizes 
\cite[Lemmas~3,~$3^\prime $]{delta}
and will be used to prove \fullref{minim}.

\begin{lemma}\label{triangle}
Suppose that the diagram
$$
\begin{CD}
X @> f >> Y
\\
@V \varphi VV @| 
\\
X^\prime @> f^\prime >> Y
\end{CD}
$$
commutes.
\begin{itemize}
\item[(1)]
If 
$\varphi _{\#} \co \pi_1(X ,x_0) 
\to \pi_1(X^\prime ,\varphi (x_0))$ is surjective,
then
\begin{itemize}
\item[(1.1)] 
$(\hat Y _{f^\prime } , p_{f^\prime })$ and
$\widehat{f^\prime } \varphi $ can be taken 
as Hopf covering and lift for $f$;
\item[(1.2)] 
the map 
$\RR (f , \hat f^\prime \varphi ) 
\to \RR (f^\prime , \hat f^\prime )$
(defined in (2) of \fullref{q})
is injective;
\item[(1.3)] 
it maps essential classes of $f$ to essential classes of $f^\prime $;
in particular, 
$\NT (f, B) \leqslant \NT (f^\prime ,B)$.
\end{itemize}
\item[(2)]
Suppose moreover that $\varphi $ has a right homotopy inverse.
If for a map $g \sim f$ the problem $g\co X\to Y\supset B$ has 
only essential preimage classes, 
then the map $g^\prime = g \chi $
is homotopic to $f^\prime $ and the problem
$g^\prime \co X^\prime \to Y \supset B$ also has 
only essential preimage classes;
in particular, $\NT (f,B)=\NT(f^\prime ,B)$.
\end{itemize}
\end{lemma}

\begin{proof}
(1.1)\qua
This follows from 
\begin{align*}
(p_{f^\prime })_{\#} 
(\pi_1(\hat Y_{f^\prime },\widehat{f^\prime } \varphi(x_0))) &=
f^\prime _{\#} (\pi_1(X^\prime ,\varphi (x_0)))
\\
(f^\prime \varphi )_{\#} (\pi_1(X,x_0)) &=
f_{\#} (\pi_1(X,x_0)).
\end{align*}
(1.2)\qua This holds because the map under consideration is given by
$$
(\hat f)^{-1} (\hat B ) = \varphi ^{-1} ((\hat f^\prime )^{-1}
(\hat B)) \mapsto (\hat f^\prime )^{-1}(\hat B).
$$
(1.3)\qua
If a preimage class $A$ of
$f\co X \to Y \supset B$ is taken to the
class $A^\prime \supset \varphi (A)$ 
of $f^\prime \co X^\prime \to Y\supset B$ 
which can be ``deleted'' by a homotopy 
$\{ f_t^\prime \}$, then the class $A$ can be ``deleted''
by the homotopy $\{ f_t^\prime \varphi \}$.

(2)\qua
Denote by $\chi $ the right homotopy inverse for $\varphi $,
ie, $\varphi \chi \sim \id _{X^\prime }$.
It is evident that
$$
g^\prime = g\chi \sim f\chi = f^\prime \varphi \chi \sim f^\prime .
$$
By (1.1) and \fullref{Hopf-c-l}, we can take 
the same Hopf covering $(\hat Y, p )$ for maps
$f$, $f^\prime $, $f\chi $ simultaneously.
Let $\hat f^\prime $ be a Hopf lift for $f^\prime $,
then 
$\hat f = \hat f^\prime \varphi $ 
and
$\hat f \chi $
are Hopf lifts for
$f$ and $f\chi $.
Lift the homotopy $f\chi \sim f^\prime $ starting at $\hat f \chi $;
let $\tilde f^\prime $ be its final map. 
By \fullref{Hopf-c-l} it is a Hopf lift for $f^\prime $.
Note that
$$
\tilde f^\prime \sim \hat f \chi = 
\hat f^\prime \varphi \chi \sim
\hat f^\prime .
$$
This homotopy and the above equalities define
(see \fullref{R-sets}) maps in the following sequence:
\begin{equation}\label{*}
\RR (f^\prime , \hat f^\prime )
\leftrightarrow
\RR (f^\prime ,\tilde f^\prime ) \leftrightarrow
\RR (f \chi , \hat f \chi )
\to 
\RR (f, \hat f)
=
\RR (f^\prime \varphi , \hat f^\prime \varphi )
\to
\RR (f^\prime , \hat f^\prime ) 
\tag{$*$}
\end{equation}
Going through this sequence, we obtain
$$
(\hat f^\prime )^{-1} (\hat B)
\mapsto
(\tilde f^\prime )^{-1} (\hat B)
\mapsto 
(\hat f \chi )^{-1}(\hat B) 
\mapsto 
\hat f ^{-1}(\hat B)
=
(\hat f^\prime \varphi )^{-1}(\hat B)
\mapsto 
(\hat f^\prime )^{-1}(\hat B) ,
$$
where $\hat B$ is an arbitrary path component 
of $p^{-1}(B)$;
this map is the identity map
$\RR (f^\prime ,\hat f^\prime )\to \RR (f^\prime ,\hat f^\prime )$.
In particular, it gives a bijection of the set of essential Nielsen
classes of $f^\prime $ onto itself.

Now suppose $g\sim f$ has only essential preimage classes.
The following diagram commutes:
$$
\begin{CD}
\RR (f\chi , \hat f \chi )
@>>>
\RR (f, \hat f)
\\
@VVV @VVV
\\
\RR (g\chi , \hat g \chi )
@>>>
\RR (g, \hat g)
\end{CD}
\quad
\quad
\quad
\begin{CD}
(\hat f \chi )^{-1}(\hat B) 
@>>> 
\hat f ^{-1}(\hat B)
\\
@VVV @VVV
\\
(\hat g \chi )^{-1}(\hat B) 
@>>>
\hat g ^{-1}(\hat B)
\end{CD}
$$
where the vertical maps are bijections defined 
by homotopies 
$\{ f_t\} \co f\sim g$ 
and
$\{ f_t \chi \} \co f\chi \sim g\chi $
(see \fullref{R-sets}),
and the horizontal ones are those of (1.2) of the present Lemma.

Suppose there exists a nonempty nonessential class of $g\chi $.
Going in the diagram up to $\RR (f\chi , \hat f\chi)$
and then to the left side of \eqref{*},
we obtain 
a nonessential class of $f^\prime $.
In contrast to it, going in the diagram right and up 
and then to the right part of \eqref{*},
we obtain 
a (nonempty) essential class of $f^\prime $.
But, as noted above, going through \eqref{*} gives an identity
map of $\RR (f^\prime , \hat f^\prime )$.
The contradiction proves the statement.
\end{proof}

The 
condition of Item (2) of \fullref{triangle}
means that the space
$X^\prime $ is dominated by the space~$X$.
Recall the definition from \cite[Chapter~I, Exercise~R, p~32]{hu}:
\begin{defin}
A space $X^\prime $ is dominated by a space~$X$
(or $X^\prime $ is a homotopy retract of $X$)
if there exist maps 
$\chi \co X^\prime \to X$ and
$\varphi \co X\to X^\prime $ 
such that 
$\varphi \chi \sim \id_{X^\prime }$.
\end{defin}

\begin{lemma}\label{equiv}
If
$
(Y,B,Y-B)
\stackrel{\psi }{\rightarrow }
(Y^\prime , B^\prime , Y^\prime  -B^\prime )$
is a homotopy equivalence,
then 
$\NT(f\co X \to Y\supset B) = 
\NT(\psi f\co X \to Y^\prime\supset B^\prime )$
for each map $f\co X\to Y$.
\end{lemma}

\begin{proof}
Let 
$\theta \co (Y^\prime , B^\prime , Y^\prime  -B^\prime ) \to (Y,B,Y-B)$ 
be a homotopy inverse for~$\psi $
(that is, $\theta \psi $ and $\psi \theta $ are homotopic to
$\id _Y$ and $\id _{Y^\prime }$ respectively, by
homotopies of maps of corresponding triples).
Let $p_f\co \hat Y\to Y$, 
$p_{\psi f}\co \hat Y^\prime \to Y^\prime $ be
Hopf coverings for $f\co X \to Y$, 
$\psi f \co X \to Y^\prime $, and let
$\smash{\hat f}$, $\smash{\widehat {\psi f}}$ be Hopf lifts for
$f$, $\psi f$.
The proof is in 6 steps.

\textbf{Step 1}\qua
There exist lifts
$\hat \psi \co \hat Y\to \hat Y^\prime $, 
$\hat \theta \co \hat Y^\prime \to \hat Y$ of maps $\psi $,
$\theta $, and
we may assume that
$\smash{\widehat {\psi f}} =\smash{\hat\psi \hat f}$.

From
$$
(p_{\psi f})_{\# } (\pi_1(\hat Y^\prime , \smash{\widehat {\psi f}} (x_0)))
=
(\psi f)_{\# } (\pi _1(X,x_0)) =
(\psi p_f )_{\# } (\pi _1(\hat Y , \hat f (x_0)))
$$
we conclude that there exists 
a lift $\hat\psi $ of $\psi $ such that
$\smash{\widehat {\psi f}} (x_0) = \hat \psi \hat f (x_0)$
and hence
$\smash{\widehat {\psi f}}  = \smash{\hat \psi \hat f}$.

Now we prove existence of $\hat \theta $. 
Since $f\sim \theta \psi f$,
from \fullref{Hopf-c-l} 
it follows that
$(\hat Y_f,p_f)$ is a Hopf covering for $\theta \psi f$.
Therefore it suffices to apply
to $\psi f $, $\theta \psi f$ what was just proved 
for $f$, $\psi f$.

\textbf{Step 2}\qua
$\hat \psi $, $\hat \theta $ are maps of triples:
$$
(\hat Y,p_f ^{-1}(B),\hat Y-p_f ^{-1}(B))
\rightleftarrows
(\hat Y^\prime , p_{\psi f}^{-1}(B^\prime ), 
\hat Y^\prime  - p_{\psi f} ^{-1}(B^\prime )). 
$$
This is easy; for $\hat \psi $, we have
$$
\hat\psi ^{-1} (p_{\psi f} ^{-1}(B^\prime ) ) =
p_{f}^{-1} (\psi ^{-1}(B^\prime )) = 
p_f^{-1} (B) ,
$$
and similarly for $\hat\theta $.

\textbf{Step 3}\qua
$\hat \psi $, $\hat \theta $ map sets of path
components of $p_f^{-1}(B)$, $p_{\psi f}^{-1}(B^\prime )$
bijectively.

Let $\{ h_t \}\co (Y,B, Y-B) \to (Y,B,Y-B)$
be a homotopy joining $h_0 = \theta \psi $ to $h_1 = \id _Y$.
Lift the homotopy 
$$ 
\{ H_t = h_t p_f \} \co
(\hat Y,p_f^{-1}(B), \hat Y - p_f^{-1}(B))
\to
(Y,B,Y-B),
$$
which joins $H_0=\theta\psi p_f $ to $H_1=p_f$,
to $\hat Y$, starting at $\hat H_0=\hat\theta\hat\psi $.
The lift is a map of triples
$$
\{ \hat H_t \} \co
(\hat Y,p_f^{-1}(B), \hat Y - p_f^{-1}(B))
\to 
(\hat Y, p_f^{-1}(B), \hat Y - p_f^{-1}(B)),
$$
and $\hat H_1$ is a covering transformation of~$p_f$.
Similarly, $\hat\psi\hat\theta $ is homotopic by
homotopy of maps of triples
$$
(\hat Y^\prime ,p_{\psi f}^{-1}(B^\prime ), 
\hat Y^\prime  - p_{\psi f}^{-1}(B^\prime ))
\to 
(\hat Y^\prime , p_{\psi f}^{-1}(B^\prime ), 
\hat Y^\prime  - p_{\psi f}^{-1}(B^\prime ))
$$
to a covering transformation of~$p_{\psi f}$.
This implies the required statement.

\textbf{Step 4}\qua
Preimage classes of the problem $f\co X\to Y \supset B$
coincide with those of
$\psi f \co X\to Y^\prime \supset B^\prime $.

Firstly, since $\psi ^{-1}(B^\prime ) = B$, we have
$
f^{-1}(B) =
(\psi f)^{-1} (B^\prime )
$.
Secondly, from Step 2 it
follows that two points
belong to the same Nielsen class of the problem
$f\co X\to Y\supset B$ iff
they belong to one Nielsen class 
of $\psi f \co X\to Y^\prime \supset B^\prime $.

\textbf{Step 5}\qua
If a preimage class $A$ 
of $f\co X\to Y\supset B$ is inessential,
then it is inessential as a preimage class of 
$\psi f\co X\to Y^\prime \supset B^\prime $.

In fact, if
a homotopy $\{ f_t \}$, with 
$f_0=f$, ``deletes'' $A$ as a preimage class
of $f\co X \to Y\supset B$, then the homotopy 
$\{ \psi f_t \}$
starts at $\psi f$ and 
``deletes'' $A$ as a preimage class of
$\psi f\co X \to Y^\prime \supset B^\prime $.

\textbf{Step 6}\qua 
Previous step shows that
$\NT (f,B)\geqslant \NT (\psi f,B^\prime )$.
Taking in this inequality 
$\psi f$, $\theta\psi f$ in place of $f$, $\psi f$,
we obtain
$\NT (\psi f,B^\prime )\geqslant \NT (\theta\psi f,B)$.
But $\theta \psi f\sim f$ implies
$\NT (\theta \psi f, B)=\NT(f,B)$, and the Lemma is
proved.
\end{proof}

Now we prove \fullref{property}.

\begin{proof}
\fullref{equiv} implies that $\NT (f,B) = \NT (\psi f,B^\prime )$.
Since $\psi f \sim f^\prime \varphi $, the last number equals
$\NT (f^\prime \varphi ,B^\prime )$. 
By (1.3) of \fullref{triangle}
(applied to the triangle of maps
$f^\prime \varphi $, $f^\prime $ and $\varphi $)
this is less than or equal to $\NT(f^\prime ,B^\prime )$,
hence
$\NT (f,B) \leqslant \NT (f^\prime , B^\prime )$
and the first statement is proved.

If $\varphi $ has a right homotopy inverse,
then by (2) of \fullref{triangle} we have
$\NT (f^\prime \varphi ,B^\prime ) = \NT (f^\prime ,B^\prime )$;
hence $\NT (f,B) = \NT (f^\prime , B^\prime )$.
\end{proof}

\begin{cor}
Suppose that the diagram
$$
\begin{CD}
X @>f >> (Y,Y-B,B) 
\\
@V \varphi VV @V \psi VV
\\
X^\prime @> f^\prime >> (Y^\prime , Y^\prime - B^\prime , B^\prime )
\end{CD}
$$
commutes up to homotopy
and $\varphi $, $\psi $ are homotopy equivalences.
Then we have
$$\NT (f\co X\to Y\supset B) 
= \NT (f^\prime \co X^\prime\to Y^\prime\supset B^\prime ).$$
\end{cor}

For roots the following corollary
is stated (without proof) in a slightly stronger form in
\cite[Theorem~(3.10)]{roots-review}.

\begin{cor}
Suppose that the diagram
$$
\begin{CD}
X @>f >> (Y,Y-B,B) 
\\
@V \varphi VV @V \psi VV
\\
X^\prime @> f^\prime >> (Y^\prime , Y^\prime - B^\prime , B^\prime )
\end{CD}
$$
commutes and $\varphi $, $\psi $ are homeomorphisms.
Then
$$\NT (f\co X\to Y\supset B) 
= \NT (f^\prime \co X^\prime\to Y^\prime\supset B^\prime ).$$
\end{cor}

\section[Other lemmas needed for proof of \ref{minim}]{Other lemmas needed for proof of \fullref{minim}}

Our proof of \fullref{minim} imitates 
those of Brooks~\cite[Theorems~1,2]{delta}. 
We also need several lemmas.

From \cite[Chapter~3, \S 21, 2A, Exercises~3,4]{ff} 
and \cite[Chapter~1, \S 6, 3C, Corollary]{ff} 
it follows that:

\begin{lemma}\label{z}
Each local coefficient system defined on the 
$2$--skeleton~$X^{(2)}$ of a polyhedron~$X$
extends (uniquely up to isomorphism) to~$X$.
\end{lemma}

\begin{lemma}{\rm \cite[Lemma 4]{delta}}\qua\label{incl}
Suppose $\mathcal C$ is a family of mutually
disjoint closed subsets of a topological space~$Z$,
and let $D=\bigcup_{C\in\mathcal C} C$.
Suppose also that each set~$C\in\mathcal C$ is both closed and open
in~$D$, and $D$ is closed in~$Z$.
Then the inclusions
$$
i_{C}\co (Z,Z-D) \hookrightarrow (Z,Z-C), \quad C\in\mathcal C,
$$
induce an isomorphism
$$
\bigg( 
\sum_{C\in\mathcal C} {i_C}_{*m}
\bigg) 
\co 
H_m (Z,Z-D) 
\cong
\sum_{C\in\mathcal C} H_m(Z,Z-C)
$$
for each $m\geqslant 0$.
\end{lemma}

\begin{lemma}\label{str-by-p}
Suppose that $Y$ is connected,
locally path connected and
semilocally simply connected space; its 
closed subspace~$B$ 
is locally path connected and $\pi_1(Y,Y-B)=0$. 
Then for each path $\alpha \co (I,0,1) \to (Y,B,Y-B)$ 
there exists a path $\beta \sim \alpha $ such that
$\beta ([0,\frac12 ])\subset B$
and $\beta ((\frac12 , 1])\subset Y-B$.
\end{lemma}

That is, not only each path in~$Y$ with end points in
$Y-B$ can be pushed off from $B$ (as it is since $B$
can be bypassed in $Y$; see \fullref{pi1}), 
but also a path with one end point 
in~$B$ and another in~$Y-B$ can be ``half-pushed off'' from~$B$.

\begin{proof}
Put $t_1 = \max \{ t\in I | \alpha ([0,t])\subset B \}$.
It is clear that $t_1 < 1$.
Let $S = \{ t\in I | t>t_1, \alpha (t)\in B\}$.
Consider two cases.

\textbf{Case 1}\qua
Suppose $t_1$ is not a limit point of $S$
(in particular, $S$ is empty).
Then there exists $t_2\in I$ such that $t_2>t_1$
and 
$\alpha ((t_1, t_2])\subset Y-B$.
Let $\alpha _1 (t) = \alpha (tt_1)$, 
$\alpha_2(t) = \alpha (tt_2 + (1-t)t_1)$,
and
$\alpha _3(t) = \alpha (t+(1-t)t_2)$, for $t\in I$.
Take a path $\beta _3\sim \alpha _3$ such that
$\beta_3 (I)\subset Y-B$. 
The path $\beta = \alpha _1 \cdot (\alpha_2\cdot \beta _3)$
has necessary properties.

\textbf{Case 2}\qua
Suppose $t_1$ is a limit point of $S$.
Let $U_1\subset Y$ be a path connected open neighbourhood of
$\alpha (t_1)$ such that each loop at $\alpha (t_1)$
in $U_1$ is homotopic (in $Y$) to a constant path.
Let $V\subset B\cap U_1$ be a path connected open (in $B$)
neighbourhood of $\alpha (t_1)$. 
Then $V = B\cap U_2$ for some open set $U_2\subset Y$. 
Denote $U=U_1\cap U_2$.
There exist $t_2,t_3\in I$ such that
$t_1<t_2 < t_3$, $\alpha ([t_1,t_2])\subset U$,
$\alpha (t_2)\in B$, and $\alpha ((t_2,t_3]) \subset Y-B$.
Denote the ``pieces''
$\alpha _1 (t) = \alpha (tt_1)$,
$\alpha _2 (t) = \alpha (tt_2 + (1-t)t_1)$,
$\alpha _3 (t) = \alpha (tt_3 + (1-t)t_2)$,
and
$\alpha _4 (t) = \alpha (t    + (1-t)t_3)$, 
for $t\in I$.
Take a path
$\beta _2 \co (I,0,1) \to (V, \alpha (t_1), \alpha (t_2))$;
we have $\alpha _2 \sim \beta _2$ (homotopy in $Y$).
Since $B$ can be bypassed in $Y$,
the path $\alpha _4$ is homotopic to some path $\beta _4$
with $\beta _4(I) \subset Y-B$.
Then $\beta = (\alpha _1 \cdot \beta _2)\cdot 
(\alpha _3\cdot \beta _4)$ is the required path.
The Lemma is proved.
\end{proof}

Just as Brooks in his proofs of
\cite[Lemmas~5,~5$^\prime $]{delta},
we will use 
Seifert--van Kampen theorem and 
relative Hurewicz theorem
in our proof of \fullref{iso}.
The next Lemma explains comprehensively,
why the spaces, to which the theorems will be applied here,
are indeed path connected. 

\begin{lemma}\label{conn}
Suppose that spaces $Y$, $B$ satisfy 
the conditions of \fullref{str-by-p}.
Let $p\co \tilde Y\to Y$ be a (arbitrary) covering.
Let $\mathcal C$ be an arbitrary family of path components
of $p^{-1}(B)$ and $D = \bigcup_{C\in\mathcal C} C$.
Then $\tilde Y - D$ is path connected.
\end{lemma}

\begin{proof}
Take arbitrary points
$\tilde y_0,\tilde y_1\in\tilde Y - D$.
To prove the existence of a path joining them, we
consider three cases.

\textbf{Case 1}\qua 
Suppose $\tilde y_0, \tilde y_1\in \tilde Y - p^{-1}(B)$.
Take a path 
$\tilde\alpha \co (I,0,1)\to (\tilde Y,\tilde y_0,\tilde y_1)$.
The path $\alpha = p (\tilde\alpha )$ in $Y$ joins the points
$p(\tilde y_0), p(\tilde y_1)\in Y-B$.
Since $B$ can be bypassed in $Y$, there exists
a path $\beta $ 
such that $\beta (I)\subset Y-B$ and $\alpha \sim \beta $.
Lift of this homotopy to $\tilde Y$ that starts at $\tilde\alpha $
gives a path over $\beta $
in $\tilde Y - p^{-1}(B)\subset \tilde Y - D$
which joins $\tilde y_0$ to $\tilde y_1$.

\textbf{Case 2}\qua 
Suppose
$\tilde y_0\in p^{-1}(B)-D$ and
$\tilde y_1 \in \tilde Y - p^{-1}(B)$. 
Join $\tilde y_0$ to $\tilde y_1$ with a path 
$\tilde \alpha $ in $\tilde Y$.
The path $\alpha = p (\tilde \alpha )$ in $Y$ begins at
$p(\tilde y_0)\in B$ and ends at
$p(\tilde y_1)\in Y-B$. 
By \fullref{str-by-p}, 
there exists a path $\beta $ such that 
$\beta \big([ 0, \frac12] \big)\subset B$, 
$\beta \big(( \frac12, 1] \big)\subset Y-B$,
and $\alpha\sim\beta $.
The lift of this homotopy to $\tilde Y$ starting at $\tilde\alpha $
gives a path $\tilde \beta $ in $\tilde Y$
between $\tilde y_0$, $\tilde y_1$.
Denote by $\tilde B^0$ the path component of
$p^{-1}(B)$ that contains $\tilde y_0$ and by $B^0$
its image under $p$, ie, the path component of $B$ which contains
$p(\tilde y_0)$.
Since $\smash{p|_{\smash{\tilde B^0}}}\co \tilde B^0 \to B^0$ is a covering,
we have
$\smash{\tilde\beta \big([ 0,\frac12  \big) \subset \tilde B^0
\subset \tilde Y - D}$ 
and 
$\smash{\tilde \beta \big(( \frac12 , 1] \big) 
\subset \tilde Y - p^{-1}(B)
\subset \tilde Y - D}$. 
So, $\tilde\beta (I)\subset \tilde Y - D$.

\textbf{Case 3}\qua 
If $\tilde y_0, \tilde y_1 \in p^{-1}(B)-D$,
take a point
$\tilde y_2 \in \tilde Y - p^{-1}(B)$
and apply Case~2 to the pairs of points
$\tilde y_0$, $\tilde y_2$ and $\tilde y_1$, $\tilde y_2$.
\end{proof}

Now we generalize and unify the lemmas
\cite[Lemmas~5,~$5^\prime $]{delta}.
The proof below contains nothing essentially new
in comparison with those of Brooks,
except reference to \mbox{\fullref{conn}}; 
and we have slightly changed the sequence of exploited ideas. 
We give the proof nevertheless, for completeness.

\begin{lemma}\label{iso}
Suppose that the conditions of \fullref{str-by-p}
are fulfilled, and
for some integer $n\geqslant 3$
we have $\pi_m(Y,Y-B) = 0$ for all
$1\leqslant m\leqslant n-1$.
Let $f\co X\to Y$ be a map
of connected, locally path connected space $X$;
and let $p\co \hat Y\to Y$ be its Hopf covering.
Let $\mathcal C$ be an arbitrary family of path components
of $p^{-1}(B)$ and $D = \bigcup_{C\in\mathcal C} C$.
Then 
$\pi_m(\hat Y,\hat Y-D,\hat y^\prime )=0$,
$\hat y^\prime \in \hat Y - D$,
for all $1\leqslant m\leqslant n-1$,
and
the inclusions
$$
i_{C}\co (\hat Y, \hat Y - D) \hookrightarrow (\hat Y, \hat Y - C),
\quad
C\in\mathcal C,
$$
induce an isomorphism
$$
\pi_n(\hat Y,\hat Y - D, \hat y^\prime ) \cong
\sum\limits_{C\in\mathcal C} 
\pi_n(\hat Y,\hat Y-C,\hat y^\prime ).
$$
\end{lemma}

\begin{proof}
Let $q\co \tilde Y\to \hat Y$ be a universal covering.
Note that \fullref{conn} implies
path connectedness of 
the spaces
$\tilde Y - q^{-1}(D)$, $\tilde Y - q^{-1}(p^{-1}(B)-D)$, and
$\tilde Y - (pq)^{-1}(B)$.
We will for brevity omit the basic points in notation
of their homotopy groups.

The composition $pq\co \tilde Y\to Y$ induces
an isomorphism
$$
\pi_m(\tilde Y,\tilde Y-(pq)^{-1}(B)) \cong
\pi_m(Y,Y-B)
$$
for each $m>0$ (for $m=1$ just a bijection)
\cite[Theorem~7.2.8]{sp}. 
So, we have 
$\pi _m (\tilde Y, \tilde Y - (pq)^{-1}(B)) = 0$
for all $1\leqslant m \leqslant n-1$.
In particular,
$$
\pi_1(\tilde Y,\tilde Y-(pq)^{-1}(B) ) =
\pi_2(\tilde Y,\tilde Y-(pq)^{-1}(B) ) = 0,
$$
so by exactness of the homotopy sequence, 
$
\pi_1(\tilde Y-(pq)^{-1}(B)) \cong \pi_1(\tilde Y)  = 0.
$

Applying the relative Hurewicz theorem 
\cite[Theorem~7.5.4]{sp}
to the pair of spaces
$(\tilde Y,\tilde Y-(pq)^{-1}(B))$,
we obtain
$H_m(\tilde Y, \tilde Y-(pq)^{-1}(B))=0$
for all $1\leqslant m\leqslant n-1$.

Represent
\begin{equation}\label{**}
\tilde Y = 
(\tilde Y - q^{-1}(D)) \cup (\tilde Y - q^{-1}(p^{-1}(B)-D)).
\tag{$**$}
\end{equation}
Note that 
$$
(\tilde Y - q^{-1}(D)) \cap (\tilde Y - q^{-1}(p^{-1}(B)-D))
=
\tilde Y - (pq)^{-1}(B).
$$
Writing the Mayer--Vietoris exact sequence
for the pairs $(\tilde Y, \tilde Y - q^{-1}(D))$ and
$(\tilde Y, \tilde Y - q^{-1}(p^{-1}(B)-D))$,
\begin{multline*}
\ldots \to
H_i (\tilde Y, \tilde Y - (pq)^{-1}(B))
\to
H_i (\tilde Y, \tilde Y - q^{-1}(D))
\oplus
H_i (\tilde Y, \tilde Y - q^{-1}(p^{-1}(B)-D))
\\
\to
H_i (\tilde Y, \tilde Y) 
\to 
H_{i-1} (\tilde Y, \tilde Y - (pq)^{-1}(B))
\to \ldots 
\end{multline*}
we obtain
$H_m(\tilde Y,\tilde Y-q^{-1}(D)) = 0$ 
for all $1\leqslant m\leqslant n-1$. 

Apply Seifert--van Kampen theorem to the 
representation \eqref{**}
of simply connected space $\tilde Y$. 
Since the intersection 
$\tilde Y - (pq)^{-1}(B)$
of the two subspaces
was proved above to be simply connected, 
the two subspaces are also simply connected.

The relative Hurewicz theorem 
applied to the pair
$(\tilde Y, \tilde Y - q^{-1}(D))$ 
of simply connected spaces gives
$\pi_m(\tilde Y,\tilde Y - q^{-1}(D))=0$,
and therefore $\pi_m(\hat Y,\hat Y-D)=0$
for all $1\leqslant m\leqslant n-1$~\cite[Theorem~7.2.8]{sp}.
The first statement is proved.

The inclusions
$
j_{C}\co (\tilde Y,\tilde Y-q^{-1}(D)) \hookrightarrow
(\tilde Y, \tilde Y - q^{-1}(C)) 
$, 
$C\in\mathcal C$,
induce by \mbox{\fullref{incl}} an isomorphism
$$
H_n (\tilde Y,\tilde Y - q^{-1}(D)) \cong
\sum\limits_{C\in\mathcal C} H_n(\tilde Y,\tilde Y-q^{-1}(C)).
$$
By naturality of the Hurewicz isomorphism
and commutativity of diagrams
$$
\begin{CD}
(\tilde Y, \tilde Y - q^{-1}(D)) @> j_{C} >>
(\tilde Y, \tilde Y - q^{-1}(C))
\\
@V q VV @V q VV
\\
(\hat Y, \hat Y - D) @> i_{C} >>
(\hat Y, \hat Y - C),
\end{CD}
$$
the map from left to right of the sequence
$$
\pi_n (\hat Y,\hat Y-D)\cong 
\pi_n(\tilde Y,\tilde Y-q^{-1}(D)) \cong
\sum\limits_{C\in\mathcal C} 
\pi_n(\tilde Y,\tilde Y-q^{-1}(C))
\cong
\sum\limits_{C\in\mathcal C} \pi_n (\hat Y,\hat Y-C),
$$
where
the central map is induced by the family $\{ j_{\mathcal C} \}$
and the outer maps by $q$,
is induced by the family $\{ i_{\mathcal C} \}$.
This finishes the proof.
\end{proof}

\section[Proof of \ref{minim}]{Proof of \fullref{minim}}

For theory of obstructions to deformations, see
Blakers and Massey \cite[(4.4)]{obstr}, Hu \cite[Chapter~VI,
  Exercise~E]{hu} and 
\cite[1]{hu-obstr}, and Schirmer \cite[\textbf{2}]{Koinz}.
We refer below to the paper \cite{Koinz} which 
contains a good summary of results 
(but only in the case of simply
connected subspace; with evident changes they hold true
in general case, for coefficients in local systems of groups).

\begin{proof}
By (2) of \fullref{triangle}
we may assume that $X$ is itself a polyhedron of dimension
less than or equal to~$n$.
Moreover, since $\pi_m(Y,Y-B)=0$ for all
$1\leqslant m\leqslant n-1$,
we assume that $f^{-1}(B)\cap X^{(n-1)} = \emptyset $
\cite[p~57]{Koinz}.
Hence we consider only the case of $\dim X=n\geqslant 3$.

Let $p\co \hat Y\to Y$ and $\hat f\co X\to \hat Y$ be 
a Hopf covering and a Hopf lift for $f$.
Let 
$\mathcal C$ be the family of all those
path components $\hat B$ of $p^{-1}(B)$ for which
${\hat f }^{-1}(\hat B)$ is either empty or an inessential
preimage class. 
Denote $D=\bigcup_{C\in\mathcal C} C$.
For $C\in\mathcal C$, 
the pullback under $\hat f$ of
the local coefficient system
$\{ \pi_n (\hat Y, \hat Y - C, y^\prime ) \}$ on
$\hat Y-C$ is
a local system on $X^{(n-1)}$; it
extends uniquely up to isomorphism to
a local system $\Gamma _{C}$ on~$X$ by \mbox{\fullref{z}}.
Similarly, the local system 
$\{ \pi_n(\hat Y,\hat Y-D, y^\prime )\}$ on
$\hat Y-D$ gives a local system $\Gamma_{D}$ on~$X$.

For each $C\in\mathcal C$ and $\hat y^{\prime } \in \hat Y - D$
the homomorphism
$$
\pi_n(\hat Y,\hat Y-D, \hat y^\prime )\to 
\pi_n(\hat Y,\hat Y-C, \hat y^\prime )
$$
induced by the inclusion
$$
i_{C}\co (\hat Y,\hat Y-D)\hookrightarrow (\hat Y,\hat Y-C)
$$
gives in its turn a homomorphism
$k_{C}\co \Gamma_{D}\to \Gamma_{C}$.
By \fullref{iso}, 
$$
\bigg( \sum\limits_{C\in\mathcal C} 
k_{C} \bigg) \co \Gamma_{D}\to
\sum\limits_{C\in\mathcal C} \Gamma_{C}
$$
is an isomorphism, therefore 
$$
\bigg( \sum\limits_{C\in\mathcal C} {k_C}_{*} \bigg) \co 
H^n(X,\Gamma_{D})\to
\sum\limits_{C\in\mathcal C} H^n(X,\Gamma_{C})
$$
is also an isomorphism.

For $C\in\mathcal C$,
let $\omega _{C}\in H^n(X,\Gamma_{C})$ be the first 
obstruction to deforming the map
$\hat f\co X\to \hat Y$ into $\hat Y-C$.
Let $\omega _{D}\in H^n(X,\Gamma_{D})$ be the first
obstruction to deforming the map
$\hat f\co X\to \hat Y$ into $\hat Y-D$.
Then $\omega_{C} = {k_C}_{*} (\omega_{D})$.
From our definition of the family $\mathcal C$ 
and
\fullref{essent}
it follows 
that
for each $C\in\mathcal C$ the map $\hat f$ 
can be deformed into $\hat Y-C$,
hence $\omega _{C} =0$; therefore
$$
\bigg( \sum\limits _{C\in\mathcal C} {k_C}_{*} \bigg) 
(\omega _{D}) =
\sum\limits_{C\in\mathcal C} \omega_{C} = 0.
$$
Since $\left( \sum_{C\in\mathcal C} {k_C}_{*} \right)$
is an isomorphism, we obtain
$\omega_{D} = 0$.
So, there exists  
a map $\smash{\hat g \sim \hat f}$ such that
$\smash{\hat g (X) \subset \hat Y-D}$ (see \cite[p~57]{Koinz}).
The map $\smash{g=p\hat g}$ is the desired one. 
\end{proof}

We give a simple corollary of our theorem.

\begin{cor}
Suppose that
$X$ is a finite-dimensional connected polyhedron,
$Y$ is a connected 
triangulated topological manifold without boundary, 
$B$ is a finite nonempty subpolyhedron of $Y$,
and $\dim X = \dim Y-\dim B \geqslant 3$. 
Then for each map $f\co X\to Y$  
there exists a map $g\sim f$ such that
each Nielsen preimage class of $g\co X\to Y\supset B$ is
topologically essential; in particular,
$\NT (f,B) = \MPCL (f,B)$.
\end{cor}

It follows easily 
from \fullref{minim} and the next easy Lemma.

\begin{lemma}\label{rel-groups}
Suppose that $Y$ is a connected triangulated 
topological manifold without boundary, and
a subset $B$ of $Y$ is a finite subpolyhedron
such that $\dim B < \dim Y$ and $Y-B$ is connected.
Then
$\pi_k (Y,Y-B) = 0$ for all 
$1\leqslant k \leqslant \dim Y - \dim B - 1$.
\end{lemma}

\begin{proof}
Let
$1\leqslant k \leqslant \dim Y - \dim B - 1$.
It suffices to prove that
for an arbitrary map $f\co (D^k, \partial D^k, x_0) \to (Y,Y-B,y_0)$
there exists a homotopy $\{ f_t \}$ of the form
$(D^k, \partial D^k)\times I \to (Y,Y-B)$, $f_0=f$,
such that $f_1 (D^k) \subset Y-B$
(see eg Postnikov \cite[p~342]{Postnikov}).
Since $B$ is compact,
there exists an integer $r$ such that 
$f(\partial D^k)$ is contained in those (closed) simplices of
the $r$--th barycentric subdivision 
of the given triangulation of $Y$, which do not intersect $B$.
We replace the original triangulation of $Y$ by its 
$r$--th barycentric subdivision, saving the same symbol $Y$.
There exists 
a simplicial approximation $g$ for $f$
such that $g$ and $f$ are homotopic by a homotopy of the
form $(D^k , \partial D^k)\times I \to (Y,Y-B)$
(see eg Spanier \cite[Theorem~3.4.8]{sp}).
Applying \cite[Theorem~5.3]{PL-topology} to subpolyhedra
$B$, $g(D^k)$, $g(\partial D^k)$ of the manifold $Y$,
we obtain an isotopy $F\co Y\times I\to Y\times I$ 
of $Y$ constant on $g(\partial D^k)$
such that $F_1 (g(D^k) - g(\partial D^k)) \cap B = \emptyset $
and hence $F_1 g (D^k) \subset Y-B$
(as usual, $F_t \co Y\to Y$ is the $t$--level map of the isotopy $Y$,
that is, the composition 
$y\mapsto F(y,t) = (y^\prime , t)\mapsto y^\prime  = F_t(y)$).
Then $\{ g_t = F_t g \} \co (D^k ,\partial D^k)\to (Y,Y-B)$ 
is a homotopy
from $g_0 = g$ to a map $g_1\co (D^k,\partial D^k)\to (Y,Y-B)$
such that $g_1(D^k) \subset Y-B$. 
This proves the Lemma.
\end{proof}

\section{Some other settings}\label{many}

Here we recall 
how another, at first sight more general, setting
\cite{fro2,fro1} 
can be reduced to the preimage problem.

Let 
$f_1,\ldots , f_r\co X\to Y$ be continuous maps,
$B\subset Y$ a nonempty subset. 
The common preimage set is
$$
\Pr (f_1,\ldots ,f_r,B) =
\{ x\in X | f_1 (x) = \ldots = f_r(x) \in B \} .
$$
In particular, if $B$ consists of just one point,
this is set of common roots; for $B=Y$
this is the coincidence set $\Coin (f_1, \ldots , f_r)$
(the possibility to reduce this setting to the preimage problem
was noted in \cite{Koinz,dk}).
This setting is equivalent to the following preimage problem:
$$
\Delta \{ f_k \} _{k=1}^{r} \co
X \to Y^r\supset \Delta B  .
$$
Indeed,
$
\Pr (f_1,\ldots , f_r,B) = (\Delta \{ f_k \} _{k=1}^{r} )^{-1}
(\Delta B)
$,
and $r$--tuples of homotopies of maps $f_1,\ldots , f_r$ 
are in one-to-one correspondence with homotopies of the
diagonal product
$\Delta \{ f_k \} _{k=1}^{r}$.
Hence the problem of estimating the numbers
\begin{align*}
\MP (f_1,\ldots , f_r,B) &=\mskip-3mu 
\min\limits_{g_1\smallsim  f_1, \ldots , g_r\smallsim f_r}
\mskip-3mu|\Pr (g_1,\ldots , g_r, B)|
\\\tag*{\hbox{and}}
\MPCL (f_1,\ldots , f_r,B) &= \mskip-3mu
\min\limits_{g_1\smallsim f_1, \ldots , g_r\smallsim f_r}
\mskip-3mu |\{ \text{Nielsen classes of } \Pr (g_1,\ldots , g_r, B) \} |
\end{align*}
(where Nielsen classes for this general case are \textit{defined}
as those of the corresponding preimage problem;
note that for coincidences of two maps this definition
agrees with the standard one) is equivalent to finding
estimates for 
$\MP (\Delta \{ f_k \} _{k=1}^r , \Delta B)$
and 
$\MPCL (\Delta \{ f_k \} _{k=1}^r , \Delta B)$.

Consequently, 
we can carry appropriate invariants and theorems from
the preimage problem over this general setting.

\subsubsection*{Acknowledgements} 
The author is grateful to Professor S\,A\,Bogatyi 
and Professor A\,V\,Zarelua
for useful discussions, and to Professor D\,L\,Gon\c{c}alves
for sending important materials.
The author is obliged to the referee for his (her)
remarks which inspired to improve this article much,
and thanks the editors Professor M\,Scharlemann and R\,Weidmann
for their attentive attitude.

\bibliographystyle{gtart}
\bibliography{link}

\end{document}